\documentclass[a4paper, 12pt, oneside]{amsart}
\usepackage[english]{babel}
\usepackage{amsfonts}
\usepackage{amssymb}
\usepackage[leqno]{amsmath}
\usepackage{amsthm}
\usepackage{amscd}
\usepackage{xy}
\usepackage{epsfig}
\usepackage{enumerate}
\usepackage[latin1]{inputenc}

\input xy
\xyoption{all}
\newtheorem{theorem}{Theorem}[section]
\newtheorem{proposition}[theorem]{Proposition}
\newtheorem{definition}[theorem]{Definition}
\newtheorem{corollary}[theorem]{Corollary}

\newtheorem{example}[theorem]{Example}
\newtheorem{lemma}[theorem]{Lemma}
\newtheorem{remark}[theorem]{Remark}

\def\Inf{\mathrm{Inf}\hspace{0.1cm}}

\title{Prime fuzzy ideals over noncommutative rings}
\author{Gabriel Navarro}
\address{Department of Computer Sciences and AI\\
 University of Granada \\
C/ El Greco s/n\\ E-51002 \\ Ceuta\\ Spain}
\email{gnavarro@ugr.es}
\author{Oscar Cortadellas}
\address{Department of Algebra\\
Faculty of Sciences\\
University of Granada \\
Avda. Fuentenueva s/n\\ E-18071\\ Granada\\ Spain}
\email{ocortad@ugr.es}
\author{F. J. Lobillo}
\address{Department of Algebra\\
ETSIIT \\
University of Granada \\
c/ Periodista Daniel Saucedo Aranda s/n\\ E-18071\\ Granada\\ Spain}
\email{jlobillo@ugr.es}
\thanks{This research has been supported
by the projects MTM2010-20940-C02-01 and FQM-266, FQM-1889 and TIC-111 (Junta de Andaluc\'{i}a
Research Groups).}

\begin{document}

\maketitle

\begin{abstract}
In this paper we introduce prime fuzzy ideals over a noncommutative ring. This notion of primeness is equivalent to level cuts being crisp prime ideals. It also generalizes the one provided by Kumbhojkar and Bapat in \cite{KumbhojkarBapat_1990}, which lacks this equivalence in a noncommutative setting. Semiprime fuzzy ideals over a noncommutative ring are also defined and characterized as intersection of primes. This allows us to introduce the fuzzy prime radical and contribute to establish the basis of a Fuzzy Noncommutative Ring Theory.
\end{abstract}

\section{Introduction}

Since the well-known paper of Rosenfeld \cite{Rosenfeld_1971} dealing with fuzzy sets of a group, many researchers have centered on giving an algebraic structure to the universe space, defining the classic algebraic topics on a fuzzy environment and studying their properties. For instance, the reader may consult the papers  \cite{Kuroki_1982} or \cite{Kuroki_1991}
about fuzzy semigroups; \cite{KumbhojkarBapat_1990}, \cite{Liu_1982}, \cite{MalikMordeson_1990}, \cite{SwamySwamy_1988} or \cite{Zhang_1988} about fuzzy ideals and fuzzy rings; \cite{LP1} or \cite{Pan_1987} about fuzzy modules; \cite{MalikMordeson_1990b} about fuzyy vector spaces; \cite{Chen_2009} about fuzzy coalgebras over a field;  \cite{Yehia_2001} about Lie algebras, and so on.
Their common methodology consists of giving some rules which link the algebraic operations with the order of the lattice where the imprecision or uncertainty is measured. On average, as the pioneer work of Zadeh \cite{Zadeh_1965} does, the unit real interval is the chosen lattice, although it is a common generalization operating over an arbitrary completely distributive lattice. Once this is done, two ways may be followed: on the one hand, the crisp algebraic structures are studied under the perspective of these new fuzzy objects whilst, on the other hand, the fuzzy objects are considered as new algebraic structures which deserve to be studied.

Focusing on the structure of ring, the early paper of Liu \cite{Liu_1982}, defining fuzzy ideals, initiated the investigation of rings by means of expanding the class of ideals with these fuzzy objects. Some years later, during the final eighties and nineties, many papers of different authors were published in order to develop a Fuzzy Ring Theory. Nevertheless, most of these authors restrict their attention to commutative rings or, simply, omit to mention that this requirement is necessary in certain cases, see Section \ref{survey}. This fact becomes surprising when we realize that noncommutative rings may be found in a wide range of knowledge areas, in which the fuzzy techniques could be applied.

For instance, the algebraic theory of error-correcting codes originally took place in the setting of vector spaces over finite fields. However, the interest in algebraic codes over finite rings grew after the realization that certain non-linear codes are actually equivalent to linear codes over the ring of integer modulo four, see \cite{Hammons_1994} or \cite{Nechaev_1991}. Recently, this results have been generalized to codes over modules over arbitrary rings \cite{greferath_2004}\cite{greferath_2006b} with emphasis in the context of codes over finite Frobenius rings. Basically they propose that ring-theoretic coding theory should use a module as alphabet, rather the ring itself.

Another example may be found in the known Representation Theory of Algebras, see for instance \cite{Simson_2006}. A directed graph $Q$ (a  \emph{quiver} in the terminology of this field) is a quadruple $(Q_0,Q_1,s,e)$ where $Q_0$ is the set of vertices, $Q_1$ is the set of arrows and for each arrow $\alpha\in
Q_1$, the vertices $s(\alpha)$ and $e(\alpha)$ are the source and the sink of $\alpha$. Hence, given a commutative field $k$, the path algebra $kQ$ of the directed graph $Q$ is the $k$-vector space generated by all possible oriented paths in $Q$. Given two paths $\alpha_1\cdots \alpha_n$ and $\beta_1\cdots \beta_m$, their composition is $(\alpha_1\cdots \alpha_n\beta_1\cdots \beta_m)$ whether $s(\beta_1)=e(\alpha_n)$, and zero otherwise. This is a noncommutative ring whose category of modules is closely related with the shape of the graph. Actually, the category of modules is equivalent to the representations of the graph.

Attending to the literature on fuzzy sets, the differential equations studied in \cite{BuckleyFeuring_2000} can be
view as fuzzy subsets of the set of differential equations. Let
$A_1(\mathbb{R})$ be the first Weyl algebra. This is identified with the
set of differential equations with polynomial coefficients (see \cite[\S
1.3]{MCRobson}). Weyl algebras are noncommutative and they are the main
tool in the algebraic treatment of (partial) differential equations.
Solutions of differential equations can be analyzed algebraically via
representations of Weyl algebras. Hence a deep knowledge of ideals in
noncommutative rings allows the study of fuzzy (partial) differential
equations with polinomial coefficients.

Finally, we would like to mention the known connection between representations of Lie algebras and representations of enveloping algebras of Lie algebras. The best
source for these results is the monograph of Dixmier
\cite{Dixmier_1996}. In particular, left modules over enveloping algebras
are the same that modules over Lie algebras. The fuzzy version of Lie
algebras can be seen in \cite{Yehia_2001}, and fuzzy modules over
noncommutative rings are studied in \cite{LP1}. Properties of fuzzy
representations and fuzzy modules over Lie algebras can be derived from
the theory of noncommutative fuzzy rings.

This paper has a double aim. On the one hand, we study the notion of primeness on fuzzy ideals clarifying relationships between various definitions appearing in the literature. As far as these authors know, due to its importance on the ring structure, primeness is the first notion under consideration in a fuzzy setting, see \cite{MukherjeeSen_1987}. Since then, some researchers have tried to redefine fuzzy primeness in order to get certain properties that a correct ``fuzzification'' of prime ideal should verify, see \cite{Kumar_1992}, \cite{KumbhojkarBapat_1993}, \cite{SwamySwamy_1988}, \cite{Zahedi_1991} or \cite{Zhang_1988}. Albeit some authors have worried about stablishing relationships between these concepts, see for instance \cite{KumbhojkarBapat_1993}, as we mentioned above, they ommit the necessity of working over a commutative ring producing certain misunderstandings. Hence, in Section \ref{survey}, we add a brief historical survey and show examples which illustrate the problem and correct some mistakes contained in previous papers.

On the other hand, it is generally accepted that the concept of fuzzy primeness considered in \cite{KumbhojkarBapat_1990} is the most appropriated. This is so, mainly, because this notion is compatible with the level cuts, i.e., it is equivalent to the level cuts being usual prime ideals. This is also congruent with the ideas of \cite{Head_1995} about deriving results concerning fuzzy structures from analogous results that concern crisp structures. Nevertheless, in a general setting, in parallel to the problems between crisp primeness in commutative and noncommutative rings, fuzzy primeness as given in \cite{KumbhojkarBapat_1990} is no longer compatible with the level cuts, see Example \ref{exD2notD4}. Hence, in Section \ref{fuzzyprime}, we propose a definition satisfying the aforementioned property, see Definition \ref{defprime}. This definition has another virtue: fuzzy semiprime ideals can be recovered as intersection of fuzzy prime ideals, see Theorem \ref{inter}. This gives the possibility of setting the fuzzy prime radical of a fuzzy ideal as the intersection of all fuzzy prime ideals containing it, see Corollary \ref{frad}. After that, hopefully we contribute to the basis of a Fuzzy Noncommutative Ring Theory which may be developed in future works.

\section{Preliminares}

All along the paper, by a fuzzy set we mean the classical concept defined in \cite{Zadeh_1965}, that is, a fuzzy set over a base set $X$ is simply a set map $\mu:X\rightarrow [0,1]$. The intersection and union of fuzzy sets is given by the point-by-point infimum and supremum. We shall use the symbols $\wedge$ and $\vee$ for denoting the infimum and supremum of a collection of real numbers.

We recall from \cite{Liu_1982} that by a left (resp. right) fuzzy ideal of an arbitrary ring $R$ we mean a fuzzy set $I:R\rightarrow [0,1]$ satisfying the following properties:
\begin{enumerate}[i)]
\item $I(x-y)\geq I(x)\wedge I(y)$ for any $x,y\in R$, i.e., it is an additive fuzzy subgroup.
\item $I(xy)\geq I(y)$ (resp. $I(xy)\geq I(x)$) for any $x,y\in R$.
\end{enumerate}
A (two-sided) fuzzy ideal $I$ is a fuzzy set which is a left and right fuzzy ideal, i.e., it satisfies the following conditions:
\begin{enumerate}[i)]
\item $I(x-y)\geq I(x)\wedge I(y)$ for any $x,y\in R$.
\item $I(xy)\geq I(x)\vee I(y)$ for any $x,y\in R$.
\end{enumerate}

From i), one may deduce that $I(0)$ is the maximum of the image of $I$, and from ii), that, if $R$ has unity, $I(1)$ is the minimum. Hence, for any $\alpha \leq I(0)$, we may consider the $\alpha$-cut, $I_\alpha$, as the subset $I_\alpha=\{\text{$x\in R$ such that $I(x)\geq \alpha$}\}$. It is easy to prove that $I$ is a fuzzy ideal if and only if $I_\alpha$ is an ideal of $R$ for any $I(1)< \alpha\leq I(0)$. The ideal $I_{I(0)}$ has a special relevancy in the literature and we shall denote it by $I_*$. Observe also that $I_{I(1)}$ is always the whole ring. Following \cite{Rosenfeld_1971}, given two fuzzy sets $A$ and $B$, we denote by $A\circ B$ the fuzzy set defined by $$(A\circ B)(x)=\bigvee_{x=x_1x_2} (A(x_1)\wedge B(x_2))$$ if $x$ can be decomposed as a product of two elements, and zero otherwise. This is done following the known Zadeh's Extension Principle \cite{Zadeh_1975}. Obviously, if $R$ has unity, any element can be decomposed, at least, trivially. For any fuzzy set $F$, the fuzzy ideal generated by $F$ will be the least ideal containing $F$, i.e., the intersection of all fuzzy ideal $I$ satisfying that $F\leq I$. We shall denote it by $\langle F \rangle$.

Throughout the paper $R$ will be an arbitrary ring with unity. For basic facts in Noncommutative Ring Theory the reader is referred, for instance, to the books \cite{Lam} or \cite{MCRobson}. Following Krull in \cite{Krull_1928} and \cite{Krull_1928b}, we recall the reader that a proper ideal $P$ of a ring $R$ is said to be prime if it verifies the following:

(*) \emph{whenever $IJ\subseteq P$ for some ideals $I$ and $J$, then $I\subseteq P$ or $J\subseteq P$}.\\
 If $R$ is commutative this is equivalent to the well-known property

(**) \emph{whenever $xy\in P$ for some $x,y\in R$, then $x\in P$ or $y\in P$}.\\
 A  ring is said to be prime if the zero ideal is prime. If $R$ is commutative, $R$ is prime if and only if it is a domain. In working in a noncommutative environment, property (**) becomes too restrictive. Actually, one may find simple non-domain rings. Then, noncommutative ring theorists call an ideal satisfying (**) \emph{completely prime} since (*) is equivalent, by \cite{Levitzki_1951} and \cite{Nagata_1951}, to the following:

 (***) \emph{whenever $xRy\subseteq P$ for some $x,y\in R$, then $x\in P$ or $y\in P$}.\\
Hence completely primeness implies primeness but the converse does not hold. In order to clarify this situation for readers non-familiar with noncommutative rings, we show the following easy exercise.

\begin{example}\label{ex}
Let $R$ be the ring of $2\times 2$ matrices over the real numbers. Let us show that the zero ideal is prime following (***), and then, following (*). Indeed, let us suppose that $X=\left(
  \begin{smallmatrix}
    a & b \\
     c & d \\
  \end{smallmatrix}
\right)$ and $Y=\left(
  \begin{smallmatrix}
    e & f \\
     g & h \\
  \end{smallmatrix}
\right)$ are two matrices such that $XTY=\left(
  \begin{smallmatrix}
    0 & 0 \\
     0 & 0 \\
  \end{smallmatrix}
\right)$ for any other matrix $T\in R$. Then, in particular,
$$\begin{array}{c}
X\left(
  \begin{smallmatrix}
    1 & 0 \\
     0 & 0 \\
  \end{smallmatrix}
\right) Y=
\left(
  \begin{smallmatrix}
    a & b \\
     c & d \\
  \end{smallmatrix}
\right)
\left(
  \begin{smallmatrix}
    1 & 0 \\
     0 & 0 \\
  \end{smallmatrix}
\right)
\left(
  \begin{smallmatrix}
    e & f \\
     g & h \\
  \end{smallmatrix}
\right)=
\left(
  \begin{smallmatrix}
    ae & af \\
     ce & cf \\
  \end{smallmatrix}
\right)=0 \Leftrightarrow a=c=0 \text{ or } e=f=0 \\
X\left(
  \begin{smallmatrix}
    0 & 1 \\
     0 & 0 \\
  \end{smallmatrix}
\right) Y=
\left(
  \begin{smallmatrix}
    a & b \\
     c & d \\
  \end{smallmatrix}
\right)
\left(
  \begin{smallmatrix}
    0 & 1 \\
     0 & 0 \\
  \end{smallmatrix}
\right)
\left(
  \begin{smallmatrix}
    e & f \\
     g & h \\
  \end{smallmatrix}
\right)=
\left(
  \begin{smallmatrix}
    ag & ah \\
     cg & ch \\
  \end{smallmatrix}
\right)=0 \Leftrightarrow a=c=0 \text{ or } g=h=0 \\
X\left(
  \begin{smallmatrix}
    0 & 0 \\
     1 & 0 \\
  \end{smallmatrix}
\right) Y=
\left(
  \begin{smallmatrix}
    a & b \\
     c & d \\
  \end{smallmatrix}
\right)
\left(
  \begin{smallmatrix}
    0 & 0 \\
     1 & 0 \\
  \end{smallmatrix}
\right)
\left(
  \begin{smallmatrix}
    e & f \\
     g & h \\
  \end{smallmatrix}
\right)=
\left(
  \begin{smallmatrix}
    be & bf \\
     de & df \\
  \end{smallmatrix}
\right)=0 \Leftrightarrow b=d=0 \text{ or } e=f=0 \\
X\left(
  \begin{smallmatrix}
    0 & 0 \\
     0 & 1 \\
  \end{smallmatrix}
\right) Y=
\left(
  \begin{smallmatrix}
    a & b \\
     c & d \\
  \end{smallmatrix}
\right)
\left(
  \begin{smallmatrix}
    0 & 0 \\
     0 & 1 \\
  \end{smallmatrix}
\right)
\left(
  \begin{smallmatrix}
    e & f \\
     g & h \\
  \end{smallmatrix}
\right)=
\left(
  \begin{smallmatrix}
    bg & bh \\
     dg & dh \\
  \end{smallmatrix}
\right)=0 \Leftrightarrow b=d=0 \text{ or } g=h=0
\end{array}
$$
Hence, a solution must verify that $X=\left(
  \begin{smallmatrix}
    0 & 0 \\
     0 & 0 \\
  \end{smallmatrix}
\right)$ or $Y=\left(
  \begin{smallmatrix}
    0 & 0 \\
     0 & 0 \\
  \end{smallmatrix}
\right)$.

Nevertheless, the zero ideal is not prime following (**), since
$$\left(
  \begin{smallmatrix}
    0 & 1 \\
     0 & 0 \\
  \end{smallmatrix}
\right)
\left(
  \begin{smallmatrix}
    0 & 1 \\
     0 & 0 \\
  \end{smallmatrix}
\right)
=\left(
  \begin{smallmatrix}
    0 & 0 \\
     0 & 0 \\
  \end{smallmatrix}
\right) \quad \text{although} \quad
\left(
  \begin{smallmatrix}
   0 & 1 \\
     0 & 0 \\
  \end{smallmatrix}
\right)
\neq
\left(
  \begin{smallmatrix}
    0 & 0 \\
     0 & 0 \\
  \end{smallmatrix}
\right)$$
\end{example}

\section{A little survey about fuzzy primeness}\label{survey}

Let us consider a ring $R$. Accordingly with crisp Ring Theory, the notion of prime fuzzy  ideal was one of the first concepts under consideration in its fuzzy version. In \cite{MukherjeeSen_1987}, it first appeared under the following definition.

\begin{definition}[D1]
A non-constant fuzzy ideal $P:R\rightarrow [0,1]$ is said to be \emph{prime} if, whenever $I\circ J\leq P$ for some fuzzy ideals $I$ and $J$, it satisfies that $I\leq P$ or $J\leq P$.
\end{definition}

This follows the standard crisp primeness once the product of ideals have been translated to the product $\circ$ of fuzzy ideals. Nevertheless, since $I\circ J$ is not necessarily  a subgroup \cite{MordesonBook}, the product $\circ$ may be replaced by $I J$, given in \cite{SwamySwamy_1988}, where
$$I J(x)=\bigvee_{x=\sum_{i}a_ib_i} \bigwedge_{i} (I(a_i)\wedge J(b_i)).$$
Furthermore, as well as for crisp Ring Theory, Zahedi in \cite{Zahedi_1991} makes use of only left or right fuzzy ideals.
\begin{definition}[D1']
A non-constant fuzzy ideal $P:R\rightarrow [0,1]$ is said to be \emph{prime} if, whenever $I J\leq P$ for some fuzzy ideals $I$ and $J$, it satisfies that $I\leq P$ or $J\leq P$.
\end{definition}

\begin{definition}[D1L-D1R]
A non-constant fuzzy ideal $P:R\rightarrow [0,1]$ is said to be \emph{prime} if, whenever $I J\leq P$ for some fuzzy left (right) ideals $I$ and $J$, it satisfies that $I\leq P$ or $J\leq P$.
\end{definition}

As pointed out in \cite{SwamySwamy_1988}, $IJ=\langle I\circ J\rangle$ the fuzzy ideal generated by $I\circ J$. Then $IJ\leq P$ if and only if $I\circ J\leq P$, and both definitions are equivalent.
In \cite[Theorem 4.9]{Zahedi_1991}, it is also proven that these are equivalent to the one-side definitions. Moreover, Malik and Mordeson \cite{MalikMordeson_1990}, completing the work of Mukherjee and Sen \cite{MukherjeeSen_1987}, and Swamy and Swamy \cite{SwamySwamy_1988} for L-fuzzy ideals, give a nice characterization of all fuzzy D1-prime ideals.

\begin{lemma}\cite{MalikMordeson_1990}\cite{SwamySwamy_1988}\label{D1prime}
A fuzzy ideal $P:R\rightarrow [0,1]$ is prime if and only $P$ is of the form
$$P(x)=\left\{
         \begin{array}{ll}
           1, & \text{if $x\in Q$}  \\
           t, & \text{otherwise}
         \end{array}
       \right. ,$$
where $Q$ is a crisp prime ideal of $R$ and $0\leq t<1$.
\end{lemma}

Albeit the former result makes fuzzy D1-primeness clear and transparent, the reader may figure out a subtle problem about its usefulness: there is no plenty more fuzzy D1-prime ideals than crisp ones. Hence, the additional information about the ring provided by fuzzy D1-prime ideals is quite reduced.

Additionally, as pointed out in \cite{Kumar_1992}, the notion seems to be too strong, since there are fuzzy ideals whose level cuts are prime, despite of they are not D1-prime. For instance, by \cite{Kumar_1992}, the fuzzy ideal of the ring of integers
$$P(x)=\left\{
         \begin{array}{ll}
           1, & \text{if $x=0$} \\
           0.8, & \text{if $x$ is even a non-zero}\\
            0.6, & \text{otherwise}
         \end{array}
       \right.$$
$P$ is not D1-prime since it is three-valued despite of each level cut is a prime ideal of $\mathbb{Z}$.

Since the first inconvenient becomes much more difficult to resolve, an evident solution to the second one is defining primeness as the property that it should verify.

\begin{definition}[D2]
A non-constant fuzzy ideal $P:R\rightarrow [0,1]$ is said to be \emph{prime} if $P_\alpha$ is prime for any $P(0)\geq \alpha >P(1)$.
\end{definition}

As the reader may think in reading this definition, it only translates our problem since now we need a characterization which makes operational D2-primeness.

In \cite{KumbhojkarBapat_1990}, Kumbhojkar and Bapat deal with fuzzy prime ideals from a element-like perspective and they state primeness looking forward similarities with the standard commutative algebra.

\begin{definition}[D3]
A non-constant fuzzy ideal $P:R\rightarrow [0,1]$ is said to be \emph{prime} if, for any $x,y\in R$, whenever $P(xy)=P(0)$, then $P(x)=P(0)$ or $P(y)=P(0)$.
\end{definition}

They give D3-primeness in order to get a fuzzy version of a well-known result in Ring Theory, namely, the quotient of a ring by a prime ideal is a prime ring, see \cite{KumbhojkarBapat_1990}. Nevertheless, since the quotient $R/P$ is exactly $R/P_*$, D3-primeness is equivalent to $P_*$ being a prime ideal of $R$. Then, trivially, D2-primeness implies D3-primeness, although the converse does not hold. For example,
$$P(x)=\left\{
         \begin{array}{ll}
           1, & \text{if $x=0$} \\
           0.8, & \text{if $x=4t$ with $t\neq 0$}\\
            0.6, & \text{otherwise}
         \end{array}
       \right.$$
is D3-prime but it is not D2-prime. For this reason, the authors give a stronger notion which they call strongly primeness.

\begin{definition}[D4]
A non-constant fuzzy ideal $P:R\rightarrow [0,1]$ is said to be \emph{prime} if, for any $x,y\in R$, $P(xy)=P(x)$ or $P(xy)=P(y)$.
\end{definition}

Obviously, D3-primeness is weaker that D4-primeness, but we are interested in the relation between D2-primeness and D4-primeness. Following \cite[Proposition]{KumbhojkarBapat_1993}, both notions are equivalent. Nevertheless, although the authors do not mention it, the proof requires commutativity since it makes use of the aforementioned property (**). The same mistake is committed in \cite[Theorem 5.3]{GuptaKantroo_1993}. In general, the equivalence becomes false as the following example shows.

\begin{example}\label{exD2notD4}
Let $R$ be the ring of $2\times 2$ matrices over the real numbers. Hence we may consider the fuzzy ideal
$$P(x)=\left\{
         \begin{array}{ll}
           1, & \text{if $x$ is the zero matrix} \\
           0, & \text{otherwise}
         \end{array}
       \right.$$
By Example \ref{ex}, the zero ideal is prime, therefore $P$ is D2-prime and D1-prime.
Nevertheless,
$$P\left(\left(   \begin{smallmatrix}
               0 & 1 \\
               0 & 0 \\
             \end{smallmatrix} \right) \left(
             \begin{smallmatrix}
               0 & 1 \\
               0 & 0 \\
             \end{smallmatrix}
           \right) \right) =P \left( \left(
             \begin{smallmatrix}
               0 & 0 \\
               0 & 0 \\
             \end{smallmatrix}
            \right)\right)=1 \quad \text{whilst} \quad P\left(\left(
             \begin{smallmatrix}
               0 & 1 \\
               0 & 0 \\
             \end{smallmatrix}
           \right) \right)=0$$
so $P$ is not D4-prime.
\end{example}

Example \ref{exD2notD4} is coherent with the circumstances which appear when dealing with prime (*) and completely prime (**) ideals over a noncommutative ring. Therefore, one may expect similar properties on fuzzy ideals.

\begin{lemma} A fuzzy ideal $P$ is D4-prime if and only if each level cut $P_\alpha$ is completely prime for all $P(0)\geq \alpha >P(1)$.
\end{lemma}
\begin{proof}
The proof is the same as in \cite[Proposition 4.2]{KumbhojkarBapat_1993}.
\end{proof}

\begin{lemma}\label{D4thenD2}
Let $R$ be an arbitrary ring with unity and $P:R\rightarrow [0,1]$ be a fuzzy ideal. If $P$ is D4-prime then it is D2-prime.
\end{lemma}
\begin{proof}
It follows from the former lemma and since completely primeness implies primeness.
\end{proof}

\begin{remark} Observe that a slightly modification of the ideal of Example \ref{exD2notD4}, simply map the zero matrix to $t$ with $0<t<1$, provides a fuzzy ideal which is D2-prime but not D1-prime.
\end{remark}

In trying to get a certain element-like definition, Zahedi in \cite{Zahedi_1991} analyzes two notions involving the so-called singletons. We recall the reader that, given an element $x\in R$ and $t\in (0,1]$, the singleton $x_t$ is defined to be the fuzzy set $x_t(x)=t$ and zero otherwise.

\begin{definition}[D0]
A non-constant fuzzy ideal $P:R\rightarrow [0,1]$ is said to be \emph{prime} if, whenever $x_ty_s\leq P$ for any singletons $x_t$ and $y_s$, then $x_t\leq P$ or $y_s\leq P$. This is called completely prime by Zahedi in \cite[Definition 2.7]{Zahedi_1991}.
\end{definition}

\begin{definition}[D$0'$]
A non-constant fuzzy ideal $P:R\rightarrow [0,1]$ is said to be \emph{prime} if, whenever $\langle x_t\rangle \langle y_s\rangle\leq P$ for any singletons $x_t$ and $y_s$, then $x_t\leq P$ or $y_s\leq P$.
\end{definition}

In \cite[Theorem 4.9]{Zahedi_1991}, D$0'$-primeness is proving to be equivalent to D1-primeness. D0-primeness is proven to be equivalent to D1-primeness in \cite{KumbhojkarBapat_1993}, see also \cite[Theorem 2.6]{MalikMordeson_1992}. Nevertheless, once again, the proof is no longer valid when $R$ is an arbitrary ring. The gap underlies in the proof of \cite[Theorem 3.5 i)]{KumbhojkarBapat_1993}, where the implication $\langle x_ry_s\rangle\subseteq P \Rightarrow \langle x_r\rangle \circ \langle y_s\rangle \subseteq P$ only can be applied when $R$ is commutative as the following example shows.

\begin{example}
Let us consider $R=\mathcal{M}_{2}(\mathbb{R})$ the ring of $2\times 2$ matrices over the real numbers. Let $P$ be the fuzzy ideal given by
$$P(x)=\left\{
         \begin{array}{ll}
           1, & \text{if $x$ is the zero matrix} \\
           0, & \text{otherwise}
         \end{array}
       \right. $$
and $x_1$ be the singleton of the element $x=\left(
  \begin{smallmatrix}
    0 & 1 \\
     0 & 0 \\
  \end{smallmatrix}
\right).$
Now, $x_1\circ x_1(z)=\bigvee_{z=z_1z_2} (x_1(z_1)\wedge x_1(z_2))\neq 0$ if and only if $z_1=z_2=x$, that is, if and only if $z=\left(
  \begin{smallmatrix}
    0 & 0 \\
     0 & 0 \\
  \end{smallmatrix}
\right)$. Then $x_1\circ x_1=x_1x_1=P$. Nevertheless, $\langle x_1\rangle=R$ and then  $\langle x_1\rangle\circ \langle x_1\rangle = R \nsubseteq P$.

Observe that $P$ is a fuzzy D1-prime, since $R$ is a prime ring and, by Lemma \ref{D1prime}, $x_1\nleq P$. Nevertheless, $x_1\nleq P$ and then $P$ is not D0-prime. This contradicts \cite[Theorem 3.5]{KumbhojkarBapat_1993} and \cite[Theorem 2.6]{MalikMordeson_1992} if we omit the commutativity condition on the ring.
\end{example}

As a summary, we recover all the facts of this section in the diagram of Figure \ref{figprimeness}. By $\complement$ we mean that the implication needs commutativity.
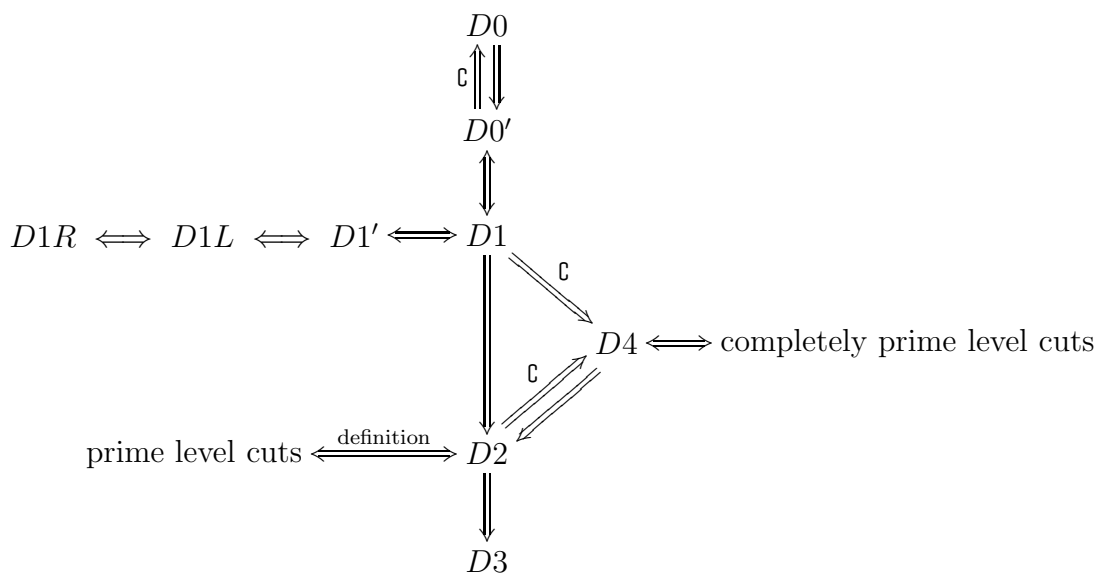
\begin{figure}[h]
$$\xymatrix@W=0.5cm{
 &   & D0  \ar@<0.7ex>@2{->}[d]& & &\\
 &   & D0' \ar@2{<->}[d]\ar@<0.7ex>@2{->}[u]^-{\complement}& & &\\
 &D1R \iff D1L \iff D1' \ar@2{<->}[r]& D1 \ar@2{->}[dr]^-{\complement} \ar@2{->}[dd] & & \\
 &  & & D4 \ar@2{<->}[r]\ar@<0.7ex>@2{->}[dl]& \text{completely prime level cuts} \\
 &  \text{prime level cuts} \ar@2{<->}[r]^-{\text{definition}} &  D2\ar@<0.7ex>@2{->}[ru]^-{\complement}\ar@2{->}[d] & \\
 &   &  D3        &  &\\
}$$
\caption{Notions of fuzzy primeness}\label{figprimeness}
    \end{figure}

As the reader may see, in a general setting, there is no characterization of primeness consistent with the level cuts.

\section{Fuzzy primeness over noncommutative rings}\label{fuzzyprime}

In this section we introduce a new definition of fuzzy prime ideals, compatible with the primeness of the level cuts, and available for any (non-necessarily commutative) ring with unity. A (crisp) ideal $P$ of $R$ is said to be prime if, whenever $IJ\leq P$ with $I$ and $J$ (left, right) ideals of $R$, then $I\leq P$ or $J\leq P$. Equivalently, $P$ is prime if and only if, whenever $xRy\subseteq P$ for some $x,y\in R$, then $x\in P$ or $y\in P$.

\begin{definition}\label{defprime} Let $R$ be an arbitrary ring with unity. A non-constant fuzzy ideal $P:R\rightarrow [0,1]$ is said to be \emph{prime} if, for any $x,y\in R$, $\mathrm{Inf} P(xRy)=P(x)\vee P(y)$.
\end{definition}

\begin{proposition}\label{charprime} Let $R$ be an arbitrary ring with unity and $P:R\rightarrow [0,1]$ be a non-constant fuzzy ideal of $R$. The following are equivalent:
\begin{enumerate}[$a)$]
\item $P$ is prime.
\item $P_\alpha$ is prime for all $P(0)\geq \alpha >P(1)$.
\item $R/P_\alpha$ is a prime ring for all $P(0)\geq \alpha >P(1)$.
\item For any fuzzy ideal $I$, if $I(xry)\leq P(xry)$ for all $r\in R$ then $I(x)\leq P(x)$ or $I(y)\leq P(y)$.
\end{enumerate}
Moreover, if $R$ is commutative, any of these statements is equivalent to $P$ being D4-prime.
\end{proposition}
\begin{proof}
$a)\Rightarrow b)$. Let $P(0)\geq \alpha >P(1)$ and $x,y\in R$ such that $xRy\subseteq P_\alpha$. Then $P(xry)\geq \alpha$ for all $r\in R$ and then $\bigwedge_{r\in R} P(xry)=P(x)\vee P(y)\geq \alpha$. Hence $P(x)\geq \alpha$ or $P(y)\geq \alpha$, so $x\in P_\alpha$ or $y\in P_\alpha$.

$b)\Rightarrow a)$. If $\mathrm{Inf} P(xRy)>P(x)\vee P(y)$ for some $x,y\in R$, consider $t=\mathrm{Inf} P(xRy)$. Hence $P_t$ is not prime since $xRy\subseteq P_t$ although $x,y\notin P_t$.

$b) \Leftrightarrow c)$ is given by the definition of prime ring.

$a)\Rightarrow d)$.  Let us suppose that $d)$ does not hold, i.e., there exist certain fuzzy ideal $I$ and $x,y\in R$ such that $I(xry)\leq P(xry)$ for all $r\in R$, and $I(x)>P(x)$ and $I(y)>P(y)$. Then $P(xry) \geq I(xry) \geq I(x)\vee I(y)> P(x)\vee P(y)$. Hence $\mathrm{Inf}P(xRy)>P(x)\vee P(y)$.

$d)\Rightarrow a)$. Suppose that $\mathrm{Inf} P(xRy)>P(x)\vee P(y)$ for some $x,y\in R$. Then there exists $t\in (0,1)$ such that $\mathrm{Inf} P(xRy)>t>P(x)\vee P(y)$. We define the ideal $I:R\rightarrow [0,1]$ given by
$$I(z)=\left\{
         \begin{array}{ll}
           P(z), & \text{if $P(z)\geq t$} \\
           t, & \text{otherwise}
         \end{array}
       \right. $$
This is a fuzzy ideal with $t<I(xry)=P(xry)$ for all $r\in R$, but $t=I(x)=I(y)>P(x)\vee P(y)$.

Finally, if $R$ is commutative $P(xry)=P(xyr)\geq P(xy)\vee P(r)\geq P(xy)$ for any $x,y,r\in R$. Then $\mathrm{Inf} P(xRy)=P(xy)$ for any $x,y\in R$. So D4-primeness condition is equivalent to $a)$.
\end{proof}

Then we may say that a ring is prime if and only if the zero ideal, viewed as a fuzzy ideal, is prime. Moreover, the ring is prime if and only if any zero-type fuzzy ideal is prime. By a zero-type ideal, we mean a fuzzy ideal $\mathbb{O}$ of the form $\mathbb{O}(0)=t$ and $\mathbb{O}(x)=s$ for $x\neq 0$, where $0\leq s<t\leq 1$. In other words, it is equivalent to the zero ideal under the following equivalence relation: two fuzzy ideals $I$ and $J$ are equivalent if $I(x)>I(y)$ if and only if $J(x)>J(y)$ for any $x,y\in R$.

\begin{remark}
Note that if $P$ is a prime fuzzy ideal, then the quotient $R/P\cong R/P_*$ is a prime ring. However, the converse does not hold since it is only equivalent to $P_*$ being prime.
\end{remark}

As well as with maximality, see \cite{MalikMordeson_1991}, we cannot assure the existence of a minimal prime fuzzy ideal. For instance, consider a minimal prime (crisp) ideal $P$ of a ring with unity $R$, which always exists, and $\chi_P$ the fuzzy prime ideal given by its characteristic map. We work over the family of fuzzy ideals $\Theta=\{(\chi_P)_t\}_{0\leq t\leq 1}$, given by $(\chi_P)_t(x)=t$ if $t\in P$, and zero otherwise. $\Theta$ is a chain of prime fuzzy ideals. Nevertheless, the minimum is the zero-constant map.

We may mend this gap, partially, making use of the equivalence relation defined above. Then we say that a prime fuzzy ideal $P$ is minimal if $P$ is equivalent to the characteristic map of a minimal prime ideal.

\begin{proposition}
Any prime fuzzy ideal contains a minimal prime fuzzy ideal.
\end{proposition}
\begin{proof}
Let $Q$ be a prime fuzzy ideal over a ring $R$. Then $Q_*$ is prime, so it contains a minimal prime ideal $P$. Then we define
$$I(x)=\left\{
         \begin{array}{ll}
           Q(0), & \text{if $x\in P$} \\
           Q(1), & \text{otherwise}
         \end{array}
       \right. $$
$I$ is a prime fuzzy ideal contained in $Q$ and equivalent to $\chi_P$. Hence it is minimal.
\end{proof}

\begin{corollary}
Let $R$ be a noetherian ring. The number of equivalence classes of minimal prime fuzzy ideals is finite.
\end{corollary}
\begin{proof}
It follows from the same property for prime crisp ideals.
\end{proof}

\section{Fuzzy semiprimeness and fuzzy prime radical}

Once primeness is achieved, following the usual way, semiprime fuzzy ideals may be defined as intersection of prime ones. Nevertheless, we follow a different, although equivalent, procedure. We use the Levitzki-Nagata Theorem characterizing semiprime ideals by an element-like property and define semiprime fuzzy ideals analogously to Definition \ref{defprime} for fuzzy prime ideals. Then we prove that semiprime fuzzy  ideals are those which can be decomposed as an intersection of prime fuzzy  ideals. This allows us to give a definition of the fuzzy prime radical which preserves the properties of a radical.

  Let us first recall the reader the scheme, similar to Figure \ref{figprimeness}, of different semiprimess notions for fuzzy ideals that appear in the literature.

\begin{definition} Let $P:R\rightarrow [0,1]$ be a non-constant fuzzy ideal over $R$. $P$ is said to be
\begin{enumerate}
\item D$0'$-semiprime if, whenever $\langle x_t\rangle^2\leq P$ for some fuzzy singleton $x_t$, then $x_t\leq P$.
\item D1-semiprime (D1L-semiprime, D1R-semiprime) if, whenever $I^2\leq P$ for some fuzzy (left, right) ideal $I$, then $I\leq P$.
\item D2-semiprime if $P_\alpha$ is semiprime for any $P(0)\geq \alpha >P(1)$.
\item D4-semiprime if $P(x^2)=P(x)$ for any $x\in R$.
\end{enumerate}
\end{definition}

Analogously to the comments of Section \ref{survey}, the reader may check that those notions are related each other as showed in Figure \ref{semiprimeness}, see \cite{DixitKumarAjmal_1992} and \cite{Zahedi_1992}. We recall that by $\complement$ we mean that the implication needs commutativity on the base ring $R$.

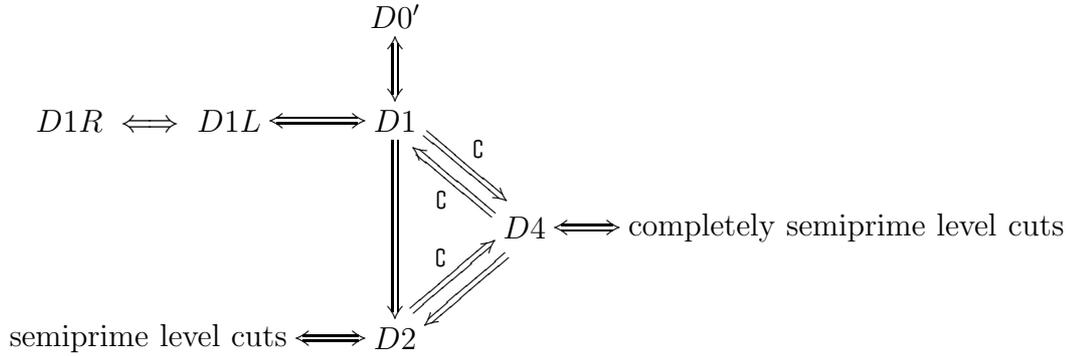
\begin{figure}[h]
$$\xymatrix{
 &   & D0' \ar@2{<->}[d]& &\\
 &D1R \iff D1L \ar@2{<->}[r]& D1 \ar@2{->}[dd] \ar@<0.7ex>@2{->}[dr]^-{\complement} & & \\
 &  & & D4 \ar@2{<->}[r]\ar@<0.7ex>@2{->}[dl]\ar@<0.7ex>@2{->}[ul]^-{\complement}& \text{completely semiprime level cuts} \\
&  \text{semiprime level cuts} \ar@2{<->}[r] &  D2\ar@<0.7ex>@2{->}[ur]^-{\complement} &
 &}$$
\caption{Notions of fuzzy semiprimeness} \label{semiprimeness}
    \end{figure}

By \cite{Levitzki_1951} and \cite{Nagata_1951}, a (crisp) ideal $P$ of $R$ is semiprime if and only if, whenever $xRx\subseteq P$ for some $x\in R$, then $x\in P$. We make use of this idea for giving the following notion of fuzzy semiprimeness.

\begin{definition} Let $R$ be an arbitrary ring with unity. A non-constant fuzzy ideal $P:R\rightarrow [0,1]$ is said to be \emph{semiprime} if $\mathrm{Inf} P(xRx)=P(x)$ for all $x\in R$.
\end{definition}

\begin{proposition}\label{charsemi}
Let $R$ be an arbitrary ring with unity and $P:R\rightarrow [0,1]$ be a non-constant fuzzy ideal of $R$. The following are equivalent:
\begin{enumerate}[$a)$]
\item $P$ is semiprime.
\item $P_\alpha$ is semiprime for all $P(0)\geq \alpha >P(1)$.
\item $R/P_\alpha$ is a semiprime ring for all $P(0)\geq \alpha >P(1)$.
\item For any fuzzy ideal $I$, if $I(xrx)\leq P(xrx)$ for all $r\in R$, then $I(x)\leq P(x)$.
\end{enumerate}
Moreover, if $R$ is commutative, any of these statements is equivalent to $P$ being D4-semiprime.
\end{proposition}
\begin{proof}
It is similar to the proof of Proposition \ref{charprime}.
\end{proof}

Then we may say that a ring is semiprime if and only if every zero-type ideal is semiprime.

In the following theorem we prove the desired characterization of semiprime ideals as  intersection of primes. For that purpose we first need a previous lemma. This is an application of
\cite[Lemma 4]{McCoy_1949}, nevertheless, in order to make the paper self-contained, we add the complete proof here.

\begin{lemma}\label{zorn} Let $R$ be a ring with unity, $P$ a semiprime ideal of $R$ and $x\in R$ such that $x\notin P$. Then there exists a prime ideal $M$ such that $P\subseteq M$ and $x\notin M$.
\end{lemma}
\begin{proof}
Since $x_0=x\notin P$ and $P$ is semiprime, there exists certain $r_0\in R$ such that $x_1=x_0r_0x_0\notin P$. Then, there exists some $r_1\in R$ such that $x_2=x_1r_1x_1\notin P$. Repeating the argument, we find a sequence of elements $X=\{x_i\}_{i\geq 0}$ in  $R$ with $P\cap X=\emptyset$. Observe that, by its construction, once an element of the sequence is in an ideal, all the next ones so is.

By Zorn's Lemma, there is an ideal $M$ maximal with respect to the property $X\cap M=\emptyset$. Clearly, $P\subseteq M$. We claim that $M$ is prime. Indeed, if $I$ and $J$ are ideals of $R$ such that $I\nsubseteq M$ and $J \nsubseteq M$, then, by the maximality of $M$, there exist some $n,m\geq 0$ such that $x_n\in I$ and $x_m\in J$. Hence, $x_t\in I\cap J$ for $t>\mathrm{max}\{n,m\}$. Now, $I\cap J\subseteq IJ$ so $IJ\nsubseteq M$
\end{proof}

\begin{theorem}\label{inter}
A fuzzy ideal is semiprime if and only if it is the intersection of prime fuzzy ideals.
\end{theorem}
\begin{proof}
Let $R$ be a ring with unity and $\{P_j\}_{j\in J}$ be a set of prime fuzzy  ideals over $R$. Then, for any $x\in R$,
$$\begin{array}{rl}
\Inf (\displaystyle\bigcap_{j\in J}P_j)(xRx)& =\displaystyle\bigwedge_{r\in R} (\displaystyle\bigcap_{j\in J}P_j)(xrx)\\
& = \displaystyle\bigwedge_{r\in R} \displaystyle\bigwedge_{j\in J} P_j(xrx) \\
& = \displaystyle\bigwedge_{j\in J}\displaystyle\bigwedge_{r\in R} P_j(xrx)\\
& = \displaystyle\bigwedge_{j\in J} \Inf P_j(xRx)\\
& \overset{\dagger}{=} \displaystyle\bigwedge_{j\in J} P_j(x) \\
& = (\displaystyle\bigcap_{j\in J}P_j)(x)
\end{array}$$
where $\dagger$ is due to the fact that $P_j$ is prime for any $j\in J$.

Conversely, let $P$ be a semiprime fuzzy ideal, we denote $$\mathcal{C}=\{\text{prime fuzzy ideals $Q$ such that $P\leq Q$}\}.$$ We claim that $\mathcal{C}\neq \emptyset$.
Since $P$ is non-constant, then $P(1)\neq P(0)$ and $P^*=\{x\in R \text{ such that } P(x)>P(1)\}$ is a proper ideal of $R$. Hence, by Zorn's Lemma, there exists a maximal ideal $M$ of $R$ containing $P^*$. We define the fuzzy ideal
$$H(x)=\left\{
         \begin{array}{ll}
           P(0), & \text{if $x\in M$} \\
           P(1), & \text{otherwise}
         \end{array}
       \right. $$
It is clear that $H$ is a prime fuzzy ideal, actually maximal in the sense of \cite[Definition 3.8]{Kumar_1992}, such that $P\leq H$.

We prove that $P=\bigcap_{Q\in \mathcal{C}}Q$. Obviously, $P\leq\bigcap_{Q\in \mathcal{C}}Q$ by definition of $\mathcal{C}$. Suppose that there exists $x\in R$ such that $P(x)< (\bigcap Q)(x)=\bigwedge Q(x)$. Let us consider two cases:

If $P(x)=P(0)$, we choose the prime fuzzy ideal $H$ defined above. Then $H(x)=P(0)=P(x)< (\bigcap Q)(x)$ and we get a contradiction.

If $P(x)<P(0)$, let $t\in (0,1)$ such that $P(x)<t< \bigwedge Q(x)$. We may suppose that $t<P(0)$. Now, $P_t$ is semiprime and $x\notin P_t$. By Lemma \ref{zorn}, there exists a prime ideal $M$ with $P_t\subseteq M$ and $x\notin M$. Then we define the fuzzy set
$$I(z)=\left\{
         \begin{array}{ll}
           P(0), & \text{if $z\in M$} \\
           t, & \text{otherwise}
         \end{array}
       \right. $$
As above, the only proper level cut is $M$, so $I$ is a prime fuzzy ideal of $R$. Furthermore, if $z\in M$, $P(z)\leq P(0)=I(z)$, and, if $z\notin M$, then $z\notin P_t$, so $P(z)<t=I(z)$. Hence $P\leq I$ and $I\in \mathcal{C}$. Although, once again, $I(x)=t< (\bigcap Q)(x)$ and we get a contradiction.
\end{proof}

In what follows we denote by $\mathrm{Rad}(I)$ the radical of an ideal $I$ of $R$, that is, the intersection of all prime ideals containing $I$.

\begin{corollary}\label{frad}
Let $R$ be a ring with unity and $I$ be a non-constant fuzzy ideal over $R$. The following fuzzy ideals coincide:
\begin{enumerate}[i)]
\item The intersection $F_1$ of all semiprime fuzzy ideals containing $I$.
\item The intersection $F_2$ of all prime fuzzy ideals containing $I$.
\item The fuzzy ideal $F_3$ given by $F_3(x)=\mathrm{Sup}\{t\in [0,1] \text{ such that } x\in \mathrm{Rad}(I_t)\}$.
\end{enumerate}
\end{corollary}
\begin{proof}
Let $\mathcal{C}$ be the set of all prime fuzzy ideals containing $I$, and $\mathcal{D}$ the set of all semiprime fuzzy ideals containing $I$.
Since every semiprime fuzzy ideal is prime, $\mathcal{C}\subseteq \mathcal{D}$ and then $F_1\leq F_2$. Let $x\in R$. Suppose that $F_3(x)< F_2(x)$ and there exists an $s\in [0,1]$ such that $F_3(x)< s<F_2(x)$. Then $x\notin \mathrm{Rad}(I_s)$, so, by Lemma \ref{zorn}, there exists a prime ideal $M$ such that $x\notin M$. We define the ideal $$P(z)=\left\{
         \begin{array}{ll}
           I(0), & \text{if $x\in M$} \\
           s, & \text{otherwise}
         \end{array}
       \right. $$
$P$ is a prime ideal since its unique proper level cut is $M$. For any $z\in M$, $I(z)\leq I(0)=P(z)$. For any $z\notin M$, $z\notin \mathrm{Rad}(I_s)$ and then $z\notin I_s$, so $I(z)<s=P(z)$. Therefore $I\leq P$ and $P\in \mathcal{C}$, but $P(x)=s<F_2(x)$ and we get a contradiction. Thus $F_2(x)\leq F_3(x)$.

Let us suppose that $F_1(x)<s<F_3(x)$ for some $s\in [0,1]$. Then $x\in \mathrm{Rad} (I_s)$ and, consequently, $xRx\subseteq \mathrm{Rad}(I_s)$. Also, there exists a semiprime fuzzy ideal $Q$ such that $Q(x)=\mathrm{Inf} \hspace{0.1cm} Q(xRx)<s$. Hence, there exists some $r\in R$ such that $Q(xrx)<s$ so $xrx\notin Q_s$. But $I\leq Q$ and $I_s\subseteq Q_s$, so $xRx\subseteq \mathrm{Rad}(I_s)\subseteq \mathrm{Rad}(Q_s)=Q_s$ and we get a contradiction. Hence $F_3(x)\leq F_1(x)$.
\end{proof}

For any non-constant fuzzy ideal $I$, we define de fuzzy prime radical of $I$, $\mathrm{FRad}(I)$, as any of the fuzzy ideals described in Corollary \ref{frad}. Hence, the following result is immediate.

\begin{corollary} $P$ is a  semiprime fuzzy ideal if and only if $\mathrm{FRad}(P)=P$.
\end{corollary}

\begin{remark} We may then define the radical of a ring as the fuzzy radical of the zero ideal. This coincide with the crisp notion of radical, since the intersection of all prime  fuzzy ideals containing the zero ideal are precisely the intersection of the characteristic maps of all minimal prime ideals of the ring. Unfortunately, this does not coincide with the radical of any other zero-type fuzzy ideal.
We only may say that these are equivalent.
\end{remark}

\begin{lemma}
For any non-constant fuzzy ideal $I$, $\mathrm{FRad}(I)(0)=I(0)$ and $\mathrm{FRad}(I)(1)=I(1)$.
\end{lemma}
\begin{proof}
Since $I\leq \mathrm{FRad}(I)$, then $I(0)\leq\mathrm{FRad}(I)(0)$ and $I(1)\leq\mathrm{FRad}(I)(1)$. Now, let $M$ be a maximal ideal of $R$ containing $I^*$, hence we define the prime fuzzy ideal $P$
 $$P(x)=\left\{
         \begin{array}{ll}
           I(0), & \text{if $x\in M$} \\
           I(1), & \text{otherwise}
         \end{array}
       \right. $$
Hence $P$ is prime and $I\leq P$, so $\mathrm{FRad}(I)(0)\leq P(0)=I(0)$.

Finally, $\mathrm{FRad}(I)(1)=\mathrm{Sup}\{t | 1\in \mathrm{Rad}(I_t)\}$. But $1\in \mathrm{Rad}(I_t)$ if and only if $\mathrm{Rad}(I_t)=R$ if and only if $I_t=R$ if and only if $t=I(1)$. Hence $\mathrm{FRad}(I)(1)=I(1)$.
\end{proof}

\begin{proposition}
Let $I$ be a non-constant fuzzy ideal over a ring $R$ with unity. Then $\mathrm{Rad}(I_t)\subseteq(\mathrm{FRad}(I))_t$ for all $I(0)\geq t > I(1)$. If $I$ satisfies the sup property, then the equality holds.
\end{proposition}
\begin{proof}
It is follows from \cite[Lemma 3.8]{MalikMordeson_1992}.
\end{proof}

\begin{proposition}
Let $P$ and $Q$ be non-constant fuzzy ideals over $R$. The following statements hold:
\begin{enumerate}[i)]
\item $\mathrm{FRad}(\mathrm{FRad}(P))=\mathrm{FRad}(P)$.
\item $\mathrm{Rad}(R/\mathrm{FRad}(R))=0$.
\item If $P\leq Q$ then $\mathrm{FRad}(P)\leq \mathrm{FRad}(Q)$.
\item $\mathrm{FRad}(P\cap Q)= \mathrm{FRad}(P)\cap \mathrm{FRad}(Q)$
\end{enumerate}
\end{proposition}
\begin{proof}
It is straightforward from the above statements.
\end{proof}

\section{Conclusions}

 The study of properties of fuzzy sets, where the base crisp set is a commutative ring, has attracted the attention of many researchers over many years. Nevertheless, many sets
are naturally endowed with two compatible operations forming a noncommutative ring. In this setting, it is easy to find examples showing that the properties proposed for commutative rings are no longer valid. Then it seems that it is necessary to study fuzzy sets over arbitrary rings.

 From this point of view, prime ideals, as structural pieces of a ring, should be the first concept under review in order to establish a well-founded fuzzy ring theory for noncommutative rings. So, in this paper, we propose a new definition of primeness for fuzzy ideals. This definition satisfies a property that seems necessary for a correct ``fuzzification'' of primeness: coherency with the level cuts. In addition, it generalizes the notion defined in \cite{KumbhojkarBapat_1990} for commutative rings, also coherent with levels cuts, albeit only when working under commutativity conditions.  Hence, fuzzy semiprime ideals are presented as an intersection of prime fuzzy ideals allowing us to define the fuzzy prime radical.

\end{document}